\documentclass{amsart}
\usepackage{mathrsfs}
\usepackage{amssymb}
\usepackage{latexsym}
\usepackage{yfonts}
\usepackage{natbib}
\bibpunct{[}{]}{;}{n}{,}{,}

\newtheorem{te}{Theorem}[section]

\theoremstyle{definition}

\theoremstyle{os}
\newtheorem{os}[te]{Remark}

\theoremstyle{prop}

\theoremstyle{lem}
\newtheorem{lem}[te]{Lemma}

\theoremstyle{coro}
\newtheorem{coro}[te]{Corollary}

\numberwithin{equation}{section}

\begin{document}

\title{On the fractional counterpart of the higher-order equations}

\author{Mirko D'Ovidio} 
\address{Department of Statistical Sciences, Sapienza University of Rome}
\email{mirko.dovidio@uniroma1.it}


\date{\today}


\subjclass[2000]{Primary }

\begin{abstract}
In this work we study the solutions to some fractional higher-order equations. Special cases in which time-fractional derivatives take integer values are also examined and the explicit solutions are presented. Such solutions can be expressed by means of the transition laws of stable subordinators and their inverse processes. In particular we establish connections between fractional and higher-order equations. 
\end{abstract}

\maketitle

\section{Introduction}
The aim of this work is to investigate the solutions to some boundary value problems involving the equation
\begin{equation} 
\left( D^{\nu_1}_{t} + D^{\nu_2}_{x} \right) u = 0, \quad x \in \Omega_0,\, t \geq 0, \quad \nu_j \in (0,1) \cup \mathbb{N},\, j=1,2 \label{eqFirst}
\end{equation}
where $\Omega_0= \Omega \cup \{0\}$ and $\Omega=(0,\infty)$. The symbol $D^{\nu}_{z}$ stands for the Riemann-Liouville fractional operator and  will be better characterized further in the text. Here the fractional powers $\nu_j$, $j=1,2$ can take both real and integer values in the interval $(0,1) \cup \mathbb{N}$. Thus, we study fractional- and higher-order equations.

Time fractional equations have been investigated by many researchers. The context in which such equations have been developed is that of anomalous diffusions or diffusions on porous media. In the works by \citet{Nig86, Wyss86, SWyss89, Koc89} the authors dealt with second-order operators in space and gave the solutions in terms of Wright or Fox functions.  More recently, in the paper by \citet{OB09} some interesting results on the explicit form of the solutions have been presented. Moreover, in the paper by \citet{BMN09ann} and the references therein the reader can find some interesting results on abstract Cauchy problems and  fractional diffusions on bounded domain.

In the present paper we deal with higher-order heat type equations in which the derivative with respect to $t$ is replaced by the fractional derivative and therefore we obtain higher-order fractional equations. Such equations have been studied by \citet{beg08} who has presented the solutions in terms of inverse Fourier transforms. In that paper, the author pointed out that such solutions can be viewed as the transition laws of compositions involving pseudo-processes and randomly varying times. Furthermore, we establish some connection between higher-order equations which lead to pseudo-processes (or higher-order diffusions) and, the time-fractional counterpart to those equations which lead to either stable or their inverse processes. In the matter of pseudo-processes (see Section \ref{sez4}) we refer to the paper by \citet{kry60, Hoc78, Funaki79, Ors91, HO94, LCH03}. 
In sections \ref{sez2},\ref{sez3} and \ref{sez4}, some preliminary results are presented whereas, the main results of this work are collected in section \ref{sez5}. In particular, we obtained the solutions to the equation \eqref{eqFirst} in the following cases:
\begin{equation*}
\begin{array}{lll}
a) & \nu_1 \in \mathbb{N},& \nu_2 \in (0,1],\\
b) & \nu_1 \in (0,1], & \nu_2 \in \mathbb{N},\\
c) & \nu_1 \in (0,1], & \nu_2  \in (0,1].
\end{array}
\end{equation*}
The special cases
\begin{equation*}
\begin{array}{lll}
a1) & \nu_1 =1,& \nu_2 = 1/n, \, n \in \mathbb{N},\\
b1) & \nu_1 = 1/n, \, n \in \mathbb{N}, & \nu_2 =1,
\end{array}
\end{equation*}
represent the fractional counterpart of the higher-order equations of order $n$. Section \ref{sez6} is devoted to the explicit representation of solutions in some particular cases involving fundamental equations of mathematical physics.

\section{Inverse processes}
\label{sez2}
Let $\varphi = \varphi(x,t)$ denote the distribution of a L\'evy process $\mathfrak{X}_t$, $t>0$ on $\mathbb{R}^n$ for which
\begin{equation}
\mathbb{E} \exp - i\xi \mathfrak{X}_t = \exp -t \Psi_{\mathfrak{X}}(\xi). \label{cfL}
\end{equation}
The L\'evy process $\mathfrak{X}_t$, $t>0$ represents the stochastic solution to the equation
$$ (D_{0+, t} - \mathcal{A})\varphi = 0, \quad \varphi \in D(\mathcal{A}) $$
where
\begin{equation*}
D(\mathcal{A}) = \left\lbrace f \in L^1_{loc}(\mathbb{R}^n)\, : \, \int_{\mathbb{R}^n} |\hat{f}(\xi)|^2 (1+ \Psi_{\mathfrak{X}}(\xi)) \, d\xi < \infty \right\rbrace. 
\end{equation*}
We have used the familiar notation in which $\hat{f}(\xi)=\mathcal{F}[f(\cdot)](\xi)$ stands for the Fourier transform of $f$. 

A stable subordinator, say $\mathfrak{H}_\nu(t)$, $t>0$, is a L\'evy process with non-negative, independent and homogeneous increments, see \citet{Btoi96} and, the $x$-Lapace transform reads as
\begin{equation}
\mathbb{E} \exp -\lambda \mathfrak{H}_\nu(t) = \mathcal{L}[h_{\nu}(\cdot, t)](\lambda) = \exp - t \lambda^\nu \label{laph}
\end{equation}
where $h_\nu$ is the density law of $\mathfrak{H}_\nu$. Straightforward calculations lead to 
\begin{equation}
\Psi_{\mathfrak{H}}(\xi) = |\xi |^{\nu} \left( -i \frac{\pi \nu}{2} \frac{\xi}{|\xi |} \right). \label{exponent}
\end{equation}
According to the literature, we define the process $\mathfrak{L}^\nu(t)$, $t>0$ as the inverse to the stable subordinator $\mathfrak{H}_\nu$ and for which $Pr\{ \mathfrak{H}_\nu(x) >t \} = Pr\{\mathfrak{ L}^{\nu}(t) < x \}$. Such an inverse process has non-negative, non-stationary and non-independent increments (see \citet{MSheff04}). The law of $\mathfrak{L}^\nu$ can be written in  terms of the Wright function
$$ W_{\alpha, \beta}(z) = \sum_{k=0}^\infty \frac{z^k}{k!\, \Gamma(\alpha k + \beta)}$$
as follows 
\begin{equation}
l_{\nu}(x,t) = t^{-\nu} W_{-\nu, 1- \nu}\left(-x t^{-\nu}\right), \quad x \in \Omega_0, \; t>0,\; \nu \in (0,1). \label{ldist}
\end{equation}
From \eqref{ldist} we immediately get the Laplace transform
\begin{equation}
\mathbb{E} \exp -\lambda \mathfrak{L}^{\nu}(t) = \mathcal{L}[l_\nu(\cdot, t)](\xi) = E_{\nu}(- \xi t^\nu) \label{lapL}
\end{equation}
in terms of the Mittag-Leffler function $E_{\alpha}(z) = E_{\alpha, 1}(z)$ where the entire function
\begin{equation*}
E_{\alpha, \beta}(z) = \sum_{k \geq 0} \frac{z^k}{\Gamma(\alpha k +\beta)}
\end{equation*}
is the generalized Mittag-Leffler for which
\begin{equation}
\int_0^\infty e^{-\lambda z} z^{\beta -1} E_{\alpha, \beta}( - \mathfrak{c} z^\alpha) \, dz = \frac{\lambda^{\alpha -\beta}}{\lambda^{\alpha} + \mathfrak{c}}. \label{lapMittagLeffler}
\end{equation}
The function $l_\nu$ is the density law of the inverse process $\mathfrak{L}_\nu$. The governing equations of both processes introduced so far can be written by means of the fractional operators \eqref{Rfracder} and \eqref{Rfracder2} as we will show in the next section. In particular, from the symbol \eqref{exponent}, we recognize that $\mathfrak{H}_{\nu}$ is a stable process (positively) totally skewed (see e.g. \citet{Zol86}). Thus, as for stable processes, we expect a fractional operator in space as well. 

We state the following relevant fact.

\begin{lem}
\label{lemma0}
The following holds
\begin{equation}
\frac{t}{x}\, l_{\nu}(t,x) = h_{\nu}(x,t), \quad x \in \Omega,\; t > 0. \label{assumef}
\end{equation}
\end{lem}
\begin{proof}
The $(x,t)-Laplace$ transforms of $h_{\nu}$ writes
\begin{equation}
\int_{0}^\infty e^{-\lambda t} \int_{0}^\infty e^{-\xi x} h_{\nu}(x,t)\, dx \, dt= \int_{0}^\infty e^{-\lambda t} e^{-t \xi^\nu} dt = \frac{1}{\lambda + \xi^\nu}. \label{llppboth}
\end{equation}
This is because of the fact that $\mathbb{E} \exp -\xi \mathfrak{H}_{\nu}(t) = \exp -t\, \xi^\nu$. From \eqref{lapMittagLeffler} we obtain the Fourier-Mellin integral
\begin{equation}
\mathcal{L}[h_{\nu}(x, \cdot)](\lambda) = \frac{1}{2\pi i} \int_{Br}  \frac{e^{\xi x} \, d\xi}{\lambda + \xi^\nu} = x^{\nu -1} E_{\nu, \nu}(-\lambda x^{\nu}).  \label{lapTh}
\end{equation}
Let us assume that formula \eqref{assumef} holds true. Thus, the $t$-Laplace transform of $h_{\nu}$ can be evaluated as follows
\begin{align*}
\mathcal{L}[h_{\nu}(x, \cdot)](\lambda) = & \int_{0}^{\infty} e^{-\lambda t} \frac{t}{x} l_{\nu}(t,x) \,dt\\
= & - \frac{1}{x} \frac{d}{d\lambda} \int_0^\infty e^{-\lambda t} l_{\nu}(t,x)\, dt\\
= & (\textrm{by } \eqref{lapL}) = - \frac{1}{x} \frac{d}{d\lambda} \, E_{\nu}(-\lambda x^{\nu})\\
= & x^{\nu -1} E_{\nu, \nu}(-\lambda x^{\nu})
\end{align*}
which coincides with \eqref{lapTh}. We can obtain the same result by considering \eqref{laph} and the fact that
\begin{equation}
\int_{0}^{\infty} e^{-\xi x} \mathcal{L}[l_\nu(x, \cdot)](\lambda)dx = \lambda^{\nu -1} \int_0^\infty e^{-\xi x}  \exp -x \lambda^\nu \, dx = \frac{\lambda^{\nu -1}}{\lambda^\nu + \xi}. \label{accordFl}
\end{equation}
The Laplace transform $\mathcal{L}[l_\nu(x, \cdot)](\lambda)$ can be easily carried out from \eqref{ldist}.
\end{proof}

\section{Fractional equations on $\Omega_0$}
\label{sez3}
It is well-known that the law $h_\nu$ of the stable subordinator $\mathfrak{H}_{\nu}$ satisfies the fractional equation 
\begin{align}
\left( D_{0+, t} + D^{\nu}_{0+, x}\right) h_{\nu} = & 0, \quad x \in \Omega,\; t> 0\label{pdehnu}
\end{align}
with initial condition $h_{\nu}(x, 0) =\delta(x)$ whereas, for the inverse process $\mathfrak{L}_{\nu}$, we have that 
\begin{align}
\left( D^{\nu}_{0+,t} + D_{0+,x} \right) l_{\nu} = & \delta(x)\, \frac{t^{-\nu}}{\Gamma(1-\nu)}, \quad x \in \Omega_0, \; t > 0 \label{fracProbL} 
\end{align}
with initial condition $l_{\nu}(x, 0) =\delta(x)$. The fractional operators appearing in  \eqref{pdehnu} and \eqref{fracProbL} are the Riemann-Liouville fractional derivatives defined in  \eqref{Rfracder} and \eqref{Rfracder2}. The initial and boundary conditions for the equation \eqref{pdehnu} can be written as 
\begin{equation}
\left\lbrace  \begin{array}{l} h_{\nu}(x, 0) = \delta(x), \quad x \in \Omega, \\
h_{\nu}(0, t) = 0, \quad t>0, \end{array} \right . \label{conditionsh}
\end{equation}
whereas, for the equation \eqref{fracProbL}, we get
\begin{align}
\left( D^{\nu}_{0+,t} + D_{0+,x} \right) l_{\nu} = & 0, \quad x \in \Omega, \; t \geq 0 \label{pdelnu}
\end{align}
subject to the initial and boundary conditions
\begin{equation}
\left\lbrace  \begin{array}{l} l_{\nu}(x, 0) = \delta(x), \quad x \in \Omega, \\
l_{\nu}(0, t) = \Phi_{\nu}(t), \quad t>0. \end{array} \right . \label{conditionsl}
\end{equation}
In \eqref{conditionsl} we considered the function
$$\Phi_{\alpha}(z) = \frac{1}{\Gamma(1-\alpha)} \, z_{+}^{-\alpha}$$ 
where $z_{+}^{-\alpha} = H(z)\, z^{-\alpha}$ with $\alpha \neq 1, 2, \ldots $ and, $H(z)$ is the Heaviside function. Since  $\Phi_{\alpha}(z) \in L^1(\mathbb{R})$ we have that
\begin{equation} \mathcal{L}[\Phi_{\alpha}(\cdot)](\zeta) = \zeta^{\alpha -1}. \label{lapHz}
\end{equation}
From the discussion made so far, the problem of finding solutions for \eqref{pdelnu} is to trace back through the study of boundary values for $l_\nu$. This approach will turn out to be useful when we study the solutions to higher-order equations which represent the higher-order counterparts of both \eqref{pdelnu} and \eqref{pdehnu}.

The problem of solving \eqref{eqFirst} with fractional powers $\nu_j \in (0,1)$, $j=1,2$ can be approached by introducing the process $\mathfrak{H}^{\nu_2}(\mathfrak{L}^{\nu_1}(t))$, $t>0$ driven by the law
\begin{equation}
\mathfrak{f}_{\nu_1, \nu_2}(x,t) = \langle h_{\nu_2}(x, \cdot), \, l_{\nu_1}(\cdot, t) \rangle , \quad x \in \Omega_{0},\; t>0,\; \nu_j \in (0,1), \; j=1,2 \label{lawf}
\end{equation}
(see e.g. \cite{ Dov2, Lanc, MLP01} and the references therein). For $\nu_1=\nu_2=\nu$ the law \eqref{lawf} takes the form $\mathfrak{f}_{\nu, \nu}(x,t)=t^{-1}f_{\nu}(t^{-1}x)$ where
\begin{equation*}
f_{\nu}(x) = \frac{1}{\pi} \frac{x^{\nu -1} \sin \pi \nu}{1+ 2 x^{\nu}\cos \pi \nu + x^{2\nu}}, \quad x \in \Omega_{0},\; t>0, \quad \nu \in (0,1)
\end{equation*}
and $\mathfrak{H}^{\nu}(\mathfrak{L}^{\nu}(t)) \stackrel{law}{=} t \times \, _1\mathfrak{H}^{\nu}(t) / \, _2\mathfrak{H}^{\nu}(t)$, $t>0$ where $\, _j\mathfrak{H}^{\nu}(t)$, $j=1,2$ are independent stable subordinators (see for example \cite{ChYor03, Dov2, Lamp63}). We notice that the ratio $\, _1\mathfrak{H}^{\nu}(t) / \, _2\mathfrak{H}^{\nu}(t)$ is independent of $t$. The governing equation of the density \eqref{lawf} (see e.g. \cite{Dov3}) is written as
\begin{equation}
(D^{\nu_1}_{0+,t} + D^{\nu_2}_{0+,x})\mathfrak{f}_{\nu_1, \nu_2}(x,t)=\delta(x)\frac{t^{-\nu_1}}{\Gamma(1-\nu_1)}, \quad x \in \Omega_{0},\; t>0 \label{eq1lamp}
\end{equation}
with $\mathfrak{f}_{\nu_1,\nu_2}(\partial \Omega_{0},t) =0$ and $\mathfrak{f}_{\nu_1, \nu_2}(x,0)=\delta(x)$ or, by considering \eqref{RCfracder}, as
\begin{equation}
\left( \frac{\partial^{\nu_1}}{\partial t^{\nu_1}} + D^{\nu_2}_{0+,x}\right)\mathfrak{f}_{\nu_1, \nu_2}(x,t)= 0, \quad x \in \Omega_{0},\; t>0. \label{eq2lamp}
\end{equation}

In the next sections we will study the remaining cases in which the powers $\nu_j$, $j=1,2$ can also take integer values. For this reason we will give a short introduction on pseudo-processes.

\section{Pseudo processes}
\label{sez4}

According to the literature we define the pseudo-process $X^{(n)}_t$, $t>0$, $n>2$, which is a Markov pseudo-process with law satisfying the higher-order heat equation
\begin{equation}
\left( \frac{\partial}{\partial t} - \kappa_{n} \frac{\partial^n}{\partial x^n}\right) v_{n}=0, \quad x \in \mathbb{R},\; t>0, \; n >2 \label{eqNord}
\end{equation}
where $\kappa_n = (-1)^{p+1}$ for $n=2p$ or $\kappa_n = \pm 1$ for $n=2p+1$. Such a process is termed ''pseudo-process'' because of the fact that the driving measure is a signed measure. We refer to the interesting work by \citet{LCH03} and the references therein for an exhaustive discussion on this topic. Here, we only recall that, for a given order $n > 2$, the solution to \eqref{eqNord} can be expressed in terms of its inverse Fourier transform
\begin{equation}
v_n(x,t) = \frac{1}{2\pi} \int_{-\infty}^{+\infty} e^{-i \zeta x + \kappa_n (-i \zeta)^n t} d\zeta \label{soleqNord}
\end{equation}
when it exists. For $n=3$ we obtain a solution in a closed form. Indeed, the solution to the third-order heat equation with $\kappa_3=\pm 1$ can be written in terms of the well-known Airy function (see, for example, \citet{LE} for information on this function) and we obtain that
\begin{equation}
v_3(x,t)=\frac{1}{\sqrt[3]{3t}}Ai\left( \frac{\mp x}{\sqrt[3]{3t}} \right), \quad x \in \mathbb{R},\; t>0.
\end{equation}
The reader can consult the paper by \citet{Ors91} for further details on this case.

The following result can be viewed as the higher-order counterpart of Lemma \ref{lemma0}.
\begin{lem}
\label{lemma1}
Let $w_j$, $j=1,2$ be two solutions of \eqref{eqNord}. We have that
\begin{equation}
w_1(x,t) = \frac{x}{t} w_2(x,t), \quad x \in \mathbb{R}, \; t>0. \label{idNordQ}
\end{equation}
\end{lem}
\begin{proof}
First, for the solution \eqref{soleqNord}, we prove that
\begin{equation}
\left( D_{0+, x}^{n -1} + \frac{\kappa_n}{n} \frac{x}{t} \right) \, v_n = 0.  \label{eighN}
\end{equation}
Let us write
$$ v_n(\zeta, t) = \mathcal{F}[v_{n}(\cdot, t)](\zeta) = \int_{\mathbb{R}} e^{i \zeta x} v_n(x,t)\, dx .$$
We have that
\begin{equation*}
\mathcal{F}[D_{0+, x}^{n -1}\, v_{n}(\cdot, t)](\zeta) = (-i \zeta)^{n-1} v_{n}(\zeta, t)
\end{equation*}
and
\begin{equation*}
\mathcal{F}[x\, v_{n}(\cdot, t)](\zeta) = \frac{1}{i} \frac{d}{d\zeta} v_n(\zeta, t) = - n \kappa_n t (-i \zeta)^{n-1} v_n(\zeta, t).
\end{equation*}
By collecting all pieces together we prove the identity \eqref{eighN}. The next step is to evaluate the following derivatives:
\begin{equation*}
\frac{\partial^n w_1}{\partial x^n} = \frac{n}{t} \frac{\partial^{n-1} w_2}{\partial x^{n-1}} + \frac{x}{t} \frac{\partial^n w_2}{\partial x^n}.
\end{equation*}
and
\begin{equation*}
\frac{\partial w_1}{\partial t} = -\frac{x}{t^2}w_2 + \kappa_n \frac{x}{t}  \frac{\partial^n w_2}{\partial x^n}
\end{equation*}
where we used the fact that $w_2$ solves \eqref{eqNord}. By summing up such derivatives we get that
\begin{align*}
\frac{\partial w_1}{\partial t} - \kappa_n \frac{\partial^n w_1}{\partial x^n} = & -\frac{x}{t^2}w_2 - \kappa_n \frac{n}{t} \frac{\partial^{n-1} w_2}{\partial x^{n-1}} \\
= & - \frac{1}{t} \left( n \kappa_n  \frac{\partial^{n-1} w_2}{\partial x^{n-1}} + \frac{x}{t}w_2 \right)
\end{align*}
where $w_2$ solves \eqref{eqNord} and thus, the identity \eqref{eighN} is in order. We obtain that
\begin{equation*}
\frac{\partial w_1}{\partial t} - \kappa_n \frac{\partial^n w_1}{\partial x^n} =0
\end{equation*}
and therefore $w_1$ solves the higher-order equation \eqref{eqNord}.
\end{proof}

\begin{os}
\normalfont
From Lemma \ref{lemma1}, by applying $m$-times the identity \eqref{idNordQ} we immediately obtain that
\begin{equation*}
w(x,t) = \frac{x^m}{t^m} v_n(x,t), \quad x \in \mathbb{R}, \; t>0, \; m \in \mathbb{N}
\end{equation*}
solves the equation \eqref{eqNord}.
\end{os}

\section{Fractional higher-order equations on $\Omega_0$}
\label{sez5}
The solutions to the fractional problem involving higher-order operators in space of a general order $n \in \mathbb{N}$ on the whole real line have been investigated by \citet{beg08}. In that paper the solutions are presented as the transition functions of pseudo-processes with randomly varying time $\mathcal{T}_{\alpha}$. In particular, see \cite[Theorem 2.3]{beg08}, the solutions to 
\begin{equation}
\left\lbrace \begin{array}{l} \frac{\partial^\alpha}{\partial t^\alpha}u(x,t) = \kappa_n \frac{\partial^n}{\partial x^n}u(x,t), \quad x \in \mathbb{R},\; t>0\\
u(x,0) = \delta(x) \end{array} \right .
\end{equation}
coincide with
$$u_{\alpha}(x,t) = \int_{0}^\infty p_n(x,s) \bar{v}_{2\alpha}(u,t)du$$
where
$\bar{v}_{2\alpha}$ is the Wright function \eqref{ldist}. In such problems only the initial conditions are required and the solutions are presented in terms of the inverse Fourier transform \eqref{soleqNord}. Indeed, the functions $p_n$ are the solutions to \eqref{eqNord}.

Our aim is to investigate fractional higher-order equations on the semi-infinite interval $\Omega_0$ with suitable boundary conditions. Due to the fact that $l_{\nu}(0^{+},t) = \Phi_{\nu}(t) < \infty$ for $t>0$, if we are looking for solutions which are symmetric, then we can extend these results to the whole real line without effort. It is enough to consider the function $l_{\nu}(|x|, t)$.

We state the following result.
\begin{te}
For $\nu \in (0,1]$, $n \in \mathbb{N}$, we have that
\begin{equation}
(D^{n}_{0-,x} - D^{\nu}_{0+, t})l_{\frac{\nu}{n}}=\, 0, \quad x \in \Omega,\; t>0 \label{eqlFH}
\end{equation}
with conditions \eqref{conditionsl} and
\begin{equation}
D^k_{0-, x} l_{\frac{\nu}{n}}(x,t)\Big|_{x=0^{+}} =  \Phi_{\frac{\nu  (k+1)}{n}}(t), \quad 0 \leq k < n. \label{derHy}
\end{equation}
\end{te}
\begin{proof}
We write the Laplace transforms
$$ \int_{0}^\infty e^{-\xi x} \int_0^\infty e^{-\lambda t} \, l_{\frac{\nu}{n}}(x,t)\, dt\, dx = l_{\frac{\nu}{n}}(\xi, \lambda).$$
Keeping in mind formulae \eqref{DDsign} and \eqref{propLint} we evaluate the Laplace transform
\begin{align*}
\mathcal{L}[D^{n}_{0-,x} l_{\frac{\nu}{n}}(\cdot, t)](\xi) = & (-1)^n \mathcal{L}[D^{n}_{0+,x} l_{\frac{\nu}{n}}(\cdot, t)](\xi)\\
= & (-1)^n \left[ \xi^n l_{\frac{\nu}{n}}(\xi, t) - \xi^{n-1} \sum_{k=0}^{n-1} (-1/\xi)^{k} \Phi_{\frac{\nu(k+1)}{n}}(t) \right]
\end{align*}
where, once again from \eqref{DDsign}, formula \eqref{derHy} has been rewritten as
$$D^k_{0+,x} l_{\nu}(x,t) \big|_{x=0^{+}} = \, (-1)^k\, \Phi_{\frac{\nu(k+1)}{n}}(t) .$$
From \eqref{lapHz} and the linearity of the Laplace transform we get that
\begin{align} 
& \int_{0}^{\infty} e^{-\lambda t} \sum_{k=0}^{n-1} (-1/\xi)^{k} \Phi_{\frac{\nu(k+1)}{n}}(t) \, dt \label{teL} \\
= & \sum_{k=0}^{n-1} (-1/\xi)^{k} \lambda^{\frac{\nu(k+1)}{n} -1} =  \lambda^{\frac{\nu}{n} - 1} \sum_{k=0}^{n -1} \left( -\lambda^{\frac{\nu}{n}} / \xi \right)^k \nonumber \\
= & \lambda^{\frac{\nu}{n} -1} \xi^{1-n} \frac{\xi^n - (-\lambda^{\frac{\nu}{n}})^{n}}{\xi + \lambda^\frac{\nu}{n}} =  \xi^{1-n} \frac{\lambda^{\frac{\nu}{n} - 1}}{\xi + \lambda^{\frac{\nu}{n}}} \left( \xi^n - (-1)^n \lambda^{\nu} \right). \nonumber
\end{align}
Thus, from \eqref{lapHH} and  \eqref{teL}, the equation \eqref{eqlFH} takes the form 
\begin{align*} 
0 = & - \lambda^{\nu} l_{\frac{\nu}{n}}(\xi, \lambda) +  (-1)^n \left[ \xi^n l_{\frac{\nu}{n}}(\xi, \lambda) - \frac{\lambda^{\frac{\nu}{n} - 1}}{\xi + \lambda^{\frac{\nu}{n}}} \left( \xi^n - (-1)^n \lambda^{\nu} \right)  \right]\\
= & \left( (-1)^n \xi^n - \lambda \right)l_{\frac{\nu}{n}}(\xi, \lambda) - \frac{\lambda^{\frac{1}{n}-1}}{\xi + \lambda^{\frac{1}{n}}} \left( (-1)^n \xi^n -  \lambda \right)
\end{align*}
and immediately we get that
\begin{align*}
l_{\frac{\nu}{n}}(\xi, \lambda) = &  \frac{\lambda^{\frac{1}{n}-1}}{\xi + \lambda^{\frac{1}{n}}}
\end{align*}
which is in accord with \eqref{accordFl}. This concludes the proof.
\end{proof}

\begin{coro}
For $\nu=1/n$, $n \in \mathbb{N}$ we have that
\begin{align}
\left( D_{0-,t} + D^{n}_{0-, x} \right) l_{\nu} = & \, 0, \quad x \in \Omega,\, t > 0
\end{align}
with conditions \eqref{conditionsl} and
\begin{equation}
D^k_{0-,x} l_{\nu}(x,t) \big|_{x=0^{+}} = \, \Phi_{\nu(k+1)}(t), \quad 0 < k < n . \label{boundcondF}
\end{equation}
is the higher-order counterpart of the equation \eqref{pdelnu}.
\end{coro}

\begin{os}
\normalfont
We notice that $\Phi_{1}(t) = 0$ and thus, for $\nu=1/n$ we have that
$$D^{n-1}_{0-,x} l_{\nu}(x,t) \big|_{x=0^{+}} = 0. $$
Furthermore, 
\begin{equation*}
D^k_{0-,x} l_{\nu}(x,t) \big|_{x=0^{+}} =  D_{0+, t}^{\nu k} \Phi_{\nu}(t) =  D_{0+, t}^{\nu k} l_{\nu}(0,t).
\end{equation*}
\end{os}

We pass to the study of the equations involving the laws of stable subordinators.
\begin{te}
For $\nu \in (0,1]$, $n \in \mathbb{N}$, we have that
\begin{equation}
(D^{n}_{0-, t} - D^{\nu}_{0+,x})h_{\frac{\nu}{n}}=0, \quad x \in \Omega,\; t>0 \label{pdehnuHigher}
\end{equation}
with initial conditions \eqref{conditionsh} and
\begin{equation}
D^{k}_{0-, t} h_{\frac{\nu}{n}}(x,t) \Big|_{t=0^{+}} = \, \Phi_{\frac{\nu k}{n}+1}(x), \quad 0 < k < n.
\end{equation}
\end{te}
\begin{proof}
First, we write
$$ \int_0^\infty e^{-\lambda t} \int_0^\infty e^{-\xi x} h_{\frac{\nu}{n}}(x,t)\, dx dt = h_{\frac{\nu}{n}}(\xi, \lambda) .$$
By evaluating the $x$-Laplace transform of the equation \eqref{pdehnuHigher} with the boundary condition $h_{\nu}(0,t) =  0$ we get that
$$\left( D^n_{0-, t} - \xi^{\nu} \right) h_{\frac{\nu}{n}}(\xi, t) = 0.$$ 
By applying standard Laplace technique, in particular from \eqref{DDsign} and \eqref{propLint}, we obtain that
\begin{equation*}
\int_{0}^{\infty} e^{-\lambda t} D^{n}_{0+,t}h_{\frac{\nu}{n}}(x,t) \,dt = \lambda^n h_{\frac{\nu}{n}}(x, \lambda) - \lambda^{n-1}\delta(x) - \lambda^{n-1} \sum_{k=1}^{n-1} (-1/\lambda)^{k} \, \Phi_{\frac{\nu k}{n}+1}(x).
\end{equation*}
From \eqref{lapHz} we get
\begin{align*} \int_0^\infty e^{-\xi x} \sum_{k=1}^{n-1} (-1/\lambda)^{k} \, \Phi_{\frac{\nu k}{n}+1}(x) dx = & \sum_{k=1}^{n-1} \left( -\xi^{\frac{\nu}{n}} /\lambda \right)^{k}
\end{align*}
and thus, the $(x,t)$-Laplace transforms of the equation \eqref{pdehnuHigher} lead to
\begin{align*}
0 = & (-1)^{n} \left[ \lambda^n h_{\frac{\nu}{n}}(\xi, \lambda) - \lambda^{n-1} - \lambda^{n-1} \sum_{k=1}^{n-1} (-\xi^{\frac{\nu}{n}}/\lambda)^k  \right] - \xi^{\nu} h_{\frac{\nu}{n}}(\xi, \lambda)\\
= & (-1)^{n} \left[ \lambda^n h_{\frac{\nu}{n}}(\xi, \lambda) - \lambda^{n-1} \sum_{k=0}^{n-1} (-\xi^{\frac{\nu}{n}}/\lambda)^k  \right] - \xi^{\nu} h_{\frac{\nu}{n}}(\xi, \lambda)\\
= & (-\lambda)^n h_{\frac{\nu}{n}}(\xi, \lambda) - \xi^{\nu} h_{\frac{\nu}{n}}(\xi, \lambda) - (-1)^n \lambda^{n-1} \sum_{k=0}^{n-1} (-\xi^{\frac{\nu}{n}}/\lambda)^k \\
= & \left( (-\lambda)^n - \xi^{\nu} \right) h_{\frac{\nu}{n}}(\xi, \lambda) - (-1)^n \lambda^{n-1} \frac{1- (-\xi^{\frac{\nu}{n}}/\lambda)^n}{ 1+ \xi^{\frac{\nu}{n}}/\lambda}\\
= & \left( (-\lambda)^n - \xi^{\nu} \right) h_{\frac{\nu}{n}}(\xi, \lambda) - \frac{(-\lambda)^n - \xi^{\nu}}{\lambda + \xi^{\frac{\nu}{n}}}.
\end{align*}
Finally, we obtain that
\begin{equation}
h_{\frac{\nu}{n}}(\xi, \lambda) = \frac{1}{\lambda + \xi^{\frac{\nu}{n}}}
\end{equation}
which is in accord with \eqref{llppboth} and this concludes the proof.
\end{proof}

From \eqref{pdehnuHigher}, for $n=1$, we immediately reobtain the equation \eqref{pdehnu}. Furthermore, a direct consequence of the previous Theorem is the following

\begin{coro}
For $\nu=1/n$, $n \in \mathbb{N}$, we have that
\begin{equation}
\left( D^n_{0-, t} + D_{0-, x} \right) h_{\nu} = \, 0, \quad x \in \Omega,\, t > 0
\end{equation}
with conditions \eqref{conditionsh} and
\begin{equation}
D^{k}_{0-, t}\, h_{\nu}(x,t) \big|_{t=0^{+}} =  \, \Phi_{\nu k+1}(x), \quad  n> k > 0.
\end{equation}
is the higher-order counterpart of the equation \eqref{pdehnu}.
\end{coro}

\begin{os}
\normalfont
The functions $l_{\nu}$ and $h_{\nu}$, for $\nu=1/n$, $n \in \mathbb{N}$ can be written as Mellin convolution of generalized gamma functions or, for $n \in 2\mathbb{N} +1$, in terms of Mellin convolution of the Modified Bessel function $K_{\alpha}$ (see \cite{Dov2, Dov3}).
\end{os}

\section{Higher-order equations: further directions}
\label{sez6}
Let $v(\cdot, t)$ be in $C(\Omega_0)$ as a function of $t$. We have that
\begin{equation*}
\langle v(x,\cdot),\, D_{0+, \cdot }\, h_{\frac{1}{n}}(\cdot ,t)\rangle = \mathcal{N}_n(x,t)  - \langle h_{\frac{1}{n}}(\cdot ,t),\, D_{0+,\cdot }\, v(x,\cdot)\rangle
\end{equation*}
where
\begin{equation*}
\mathcal{N}_n(x,t) = v(x,s)\, h_{\frac{1}{n}}(s,t)\Big|_{s \in \partial \Omega_0}.
\end{equation*}
For $n>1$, the conditions \eqref{conditionsh} hold true and then we get that $\mathcal{N}_n(x,t) \equiv 0$ and
\begin{equation*}
\langle v(x, \cdot),\, D_{0+, \cdot}\, h_{\frac{1}{n}}(\cdot,t)\rangle =  \langle h_{\frac{1}{n}}(\cdot,t),\, D_{0-,\cdot}\, v(x,\cdot)\rangle.
\end{equation*}
Thus, the stochastic solution to the equation
\begin{equation}
(D^{n}_{0-,t} + D^{m}_{0-,x}) \mathfrak{u}_{m,n} =0, \quad x \in \Omega,\; t>0, \quad m, n >1 \label{eqHordmn}
\end{equation} 
is given by the process $\mathfrak{L}^{\frac{1}{m}}(\mathfrak{H}^{\frac{1}{n}}(t))$, $t>0$ with law 
$$\mathfrak{u}_{m,n}(x,t) = \langle l_{\frac{1}{m}}(x,\cdot),\, h_{\frac{1}{n}}(\cdot,t)\rangle.$$
From the fact that $l_{\nu}(x,s) \stackrel{\nu \to 1}{\longrightarrow} \delta(x-s)$ and $h_{\nu}(s,t) \stackrel{\nu \to 1}{\longrightarrow} \delta(s - t)$ we obtain that $\mathfrak{u}_{m, 1}=l_{1/m}$ and $\mathfrak{u}_{1,n}=h_{1/n}$. As a direct consequence we also obtain that
\begin{equation*}
\lim_{n \to 1} \mathcal{N}_{n}(x,t) = \psi(x,t), \quad x \in \Omega_{0},\; t \geq 0
\end{equation*}
which must be understood in the sense of distribution.

\subsection{$m=2$}
The case $m=2$ leads to the process $|B(\mathfrak{H}^{\frac{1}{n}}(t))|$, $t>0$ where $|B|$ is the reflecting Brownian motion. From the subordination principle we have that the law of $|B(\mathfrak{H}^{\frac{1}{n}})|$ coincides with the folded law of a $2/n$-stable process with characteristic function \eqref{cfL} and $\Psi(\xi) = |\xi |^{2/n}$. Thus, for $n=2$ (and $m=2$), we obtain the folded Cauchy process driven by the Laplace equation. 

\subsection{$m=3$}
For $m=3$ we focus on the following few cases: 
\begin{align*}
(D_{0+,t} + D^3_{0+,x}) u_1 = 0,\\
(D^2_{0+,t} - D^3_{0+,x})u_2=0,\\
(D^{3}_{0+,t} + D^{3}_{0+,x})u_3=0.
\end{align*}
The explicit solutions to the equations above are listed below:
\begin{equation*}
u_1(x,t) = \frac{1}{\pi}\sqrt{\frac{x}{t}} K_{\frac{1}{3}}\left( \frac{2}{3^{3/2}} \frac{x^{3/2}}{\sqrt{t}}\right)
\end{equation*}
(which is the solution of \eqref{eqNord} on the positive real line) where (see \citet[formula 5.7.2]{LE})
$$K_{\nu}(z) = \frac{\pi}{2} \frac{I_{-\nu}(z) - I_{\nu}(z)}{\sin \nu \pi}$$ 
is the modified Bessel function of the second kind (or Macdonald's function) and
$$ I_{\nu}(z) = \sum_{k=0}^\infty \frac{(z/2)^{2k + \nu}}{k!\, \Gamma(k + \nu + 1)}$$
is the modified Bessel function of the first kind (see \cite[formula 5.7.1]{LE}); 
\begin{equation*}
u_2(x,t) = \frac{1}{4} \sqrt{\frac{3}{x \pi}} \exp\left( \frac{x^3}{3^{3}t^2} \right) \mathcal{W}_{-\frac{1}{2}, \frac{1}{6}}\left( \frac{2 x^3}{3^{3}t^2} \right)
\end{equation*}
where (see \citet[formula 9.13.16]{LE})
$$\mathcal{W}_{\alpha, \beta}(z) = z^{\beta + 1/2} e^{-\frac{z}{2}} U(1/2 - \alpha + \beta, 2\beta +1; z)$$ 
is the Whittaker function ($U$ is the confluent hypergeometric function of the second kind);
\begin{equation*}
u_3(x,t) = \frac{2t}{3^{3/2} \pi} \frac{x-t}{x^3 - t^3}.
\end{equation*}
The solution $u_2$ can be obtained by considering the integral
\begin{align*}
u_2(x,t) = \int_{0}^{\infty} u_1(x,s)\, \frac{t\, e^{-\frac{t^2}{4s}}}{\sqrt{4\pi s^3}}\, ds
\end{align*}
and the formula 6.631 of \citet{GR}. The solution $u_3$ comes out from the integral
\begin{equation*}
u_3(x,t) = \int_{0}^{\infty} u_1(x,s)\, \frac{1}{3\pi} \frac{t^{3/2}}{s^{3/2}}K_{\frac{1}{3}}\left(\frac{2}{3^{3/2}}  \frac{t^{3/2}}{\sqrt{s}} \right)\, ds
\end{equation*}
and the fact that
\begin{equation*}
\int_0^\infty s \, K_{\nu}(ys)\,K_{\nu}(zs)\, ds = \frac{\pi (yz)^{-\nu}(y^{2\nu} - z^{2\nu})}{2\sin \pi \nu\, (y^2 - z^2)}, \quad \Re\{ y+z\}>0,\; |\Re\{ \nu \}|<1
\end{equation*}
(see \cite[formula 6.521]{GR}). We observe that 
$$u_3(x,t) = \frac{2}{3^{3/2} \pi} \frac{t}{x^2 + xt + t^2}$$
which agrees, in some sense, with \eqref{lawf} for $\nu_1=\nu_2$. 

\subsection{$m=4$}
For $n=1$, the solution to \eqref{eqHordmn} is the law of the folded iterated Brownian motion $|B(|B(t)|)|$, $t>0$. This is because of the fact that $l_{1/4} = l_{1/2} \circ l_{1/2}$ is the density law of $\mathfrak{L}^{1/2}(\mathfrak{L}^{1/2}(t))$, $t>0$. The solution $y(x,t) = \mathfrak{u}_{4,1}(|x|, t)$, $x \in \mathbb{R}$, $t>0$ has been thoroughly studied (see e.g. \citet{AZ01}) and solves the generalized equation
\begin{equation*} 
\frac{\partial y}{\partial t} = \frac{\partial^4 y}{\partial x^4} + \frac{t^{-1/2}}{\sqrt{\pi}} \delta^{\prime \prime}, \quad x \in \mathbb{R},\; t\geq 0
\end{equation*}
where $\langle \delta^{\prime \prime}, \phi \rangle = \phi^{\prime \prime}(0)$ for $\phi \in C_c^{\infty}(\mathbb{R})$. In this case we have that 
$$\mathcal{N}_{n}(x,t) \stackrel{n \to 1}{\longrightarrow} \delta^{\prime \prime}(x) / \sqrt{\pi t}.$$

For $n=2$ we have that $r(x,t) = \mathfrak{u}_{4,2}(|x|, t)$ solves the equation of vibrations of rods on $\mathbb{R}$ and the corresponding process is $B(|S_{1/4}(t)|)$ where $B$ is the standard Brownian motion and $S_{\alpha}$ is the L\'evy process with $\Psi_{S}(\xi) = |\xi |^{\alpha}$ and $\alpha \in (0,2]$ (that is, an $\alpha$-stable symmetric process).

\appendix

\section{Fractional and Higher-order derivatives}

We recall the Riemann-Liouville fractional derivatives
\begin{equation}
D^{\alpha}_{0+,z}f(z) = \frac{1}{\Gamma\left( n-\alpha \right)} \frac{d^n}{d z^n} \int_0^z (z-s)^{n-\alpha -1} f(s) \, ds, \quad z \in \mathbb{R}_{+} \label{Rfracder}
\end{equation}
and
\begin{equation}
D^{\alpha}_{0-,z}f(z) = - \frac{1}{\Gamma(n-\alpha)} \frac{d^n}{dz^n} \int_z^{\infty} (s - z)^{n-\alpha -1} f(s) \, ds, \quad z \in \mathbb{R}_{+}  \label{Rfracder2}
\end{equation}
which hold for $n-1 < \alpha < n$, $n \in \mathbb{N}$. Formulae \eqref{Rfracder} and \eqref{Rfracder2} are usually refered to the right and left Riemann-Liouville fractional derivatives and become an ordinary derivative for $\alpha =n \in \mathbb{N}$. Indeed,
\begin{equation}
D^n_{0+,z} = (-1)^n D^n_{0-, z} = \partial^n/\partial z^n. \label{DDsign}
\end{equation}
For the sake of simplicity we will write $D_{0+,z}$ instead of $D^1_{0+,z}$. Furthermore, we recall that 
\begin{equation}
D^\alpha_{0+,z} f(z) = \frac{d^\alpha f}{d z^\alpha}(z) + \sum_{k=0}^{n-1} D^{k}_{0+, z} f(z) \Bigg|_{z=0^+} \frac{z^{k - \alpha}}{\Gamma(k - \alpha +1)}, \quad n-1 < \alpha < n \label{RCfracder}
\end{equation}
(see \citet{GM97,KST06, SKM93}) where
\begin{equation}
\frac{d^\alpha}{d z^\alpha}f(z) = \frac{1}{\Gamma\left( n-\alpha \right)} \int_0^z (z-s)^{n-\alpha -1} \frac{d^n f}{d s^n}(s) \, ds, \quad n-1 < \alpha < n \label{Cfracder}
\end{equation}
is the Dzerbayshan-Caputo fractional derivatives. The Laplace transform of \eqref{RCfracder} is written as
\begin{equation}
\mathcal{L}[D^\alpha_{0+,z} f(\cdot)](\zeta) = \zeta^{\alpha} \mathcal{L}[ f(\cdot)](\zeta). \label{lapHH}
\end{equation}
For the sake of completeness we also recall the Laplace transform
\begin{equation}
\mathcal{L}[D^{n}_{0+, z}f(\cdot)](\zeta) = \zeta^n \mathcal{L}[f(\cdot)](\zeta) - \sum_{k=0}^{n-1} \zeta^{n-k -1} D^{k}_{0+, z}f(z)\bigg|_{z=0^{+}}. \label{propLint}
\end{equation}

\end{document}